\documentclass[12pt,reqno]{amsart}
\usepackage{latexsym,amsmath,amsfonts,amssymb,amsthm,color}
 \theoremstyle{plain}
\newtheorem{Thm}{Theorem}
\newtheorem{Lem}{Lemma}

\theoremstyle{definition}

\theoremstyle{remark}

\def\1{{\bf 1}}

\begin{document}
\title{Symmetric identities on Bernoulli polynomials}
\author{Amy M. Fu}
\email{fu@nankai.edu.cn}
\address{Center for Combinatorics, LPMC-TJKLC, Nankai University, Tianjin 300071, People's Republic of China}
\author{Hao Pan}
\email{haopan79@yahoo.com.cn}
\address{
Department of Mathematics, Shanghai Jiaotong University, Shanghai
200240, People's Republic of China}
\author{Fan Zhang}
\email{zhangfan03@mail.nankai.edu.cn}
\address{Center for Combinatorics, LPMC-TJKLC,  Nankai University, Tianjin 300071, People's Republic of
China} 
\thanks{This work was supported by the 973
Project, the PCSIRT Project of the Ministry of Education, the
Ministry of Science and Technology, and the National Science
Foundation of China.}
 \maketitle

\vskip 10pt

\begin{abstract} In this paper, we obtain a generalization of
an identity due to Carlitz on Bernoulli polynomials. Then we use
this generalized formula to derive two symmetric identities which
reduce to some known identities on Bernoulli polynomials and
Bernoulli numbers, including the Miki identity. \end{abstract}

\section{Introduction}
\setcounter{Lem}{0}\setcounter{Thm}{0}\setcounter{Cor}{0}
\setcounter{equation}{0}

 The Bernoulli polynomials $B_n(x),\
n=1,2,\ldots$ are given by the generating function:
$$
\sum_{n=0}^\infty\frac{B_n(x)t^n}{n!}=\frac{te^{xt}}{e^t-1}.
$$
In particular, we call $B_n:=B_n(0)$ the $n$-th Bernoulli number.

In \cite{Miki78}, Miki discovered the following identity on
Bernoulli numbers:
\begin{equation}
\label{miki}
\sum_{k=2}^{n-2}\frac{B_kB_{n-k}}{k(n-k)}=\sum_{k=2}^{n-2}\binom{n}{k}\frac{B_kB_{n-k}}{k(n-k)}+2H_n\frac{B_n}{n},
\end{equation}
where
$$
H_n:=\sum_{i=1}^n\frac{1}{i}
$$
is the $n$-th harmonic number. Later several different proofs of
(\ref{miki}) were found by Shirantani and Yokoyama
\cite{ShirantaniYokoyama82}, Gessel \cite{Gessel05}, Dunne and
Schubert \cite{DunneSchubert04}. Furthermore, in the same paper
Dunne and Schubert also proved a similar identity conjectured by
Matiyasevich \cite{Matiyasevich97}:
\begin{equation}
\label{matiyasevich}
(n+2)\sum_{k=2}^{n-2}B_kB_{n-k}=2\sum_{k=2}^{n-2}\binom{n+2}{k}B_kB_{n-k}+n(n+1)B_n,
\end{equation}
On the other hand, using a new difference-differential method, Pan
and Sun \cite{PanSun06} established the following generalizations of
\eqref{miki} and \eqref{matiyasevich} for Bernoulli polynomials:
\begin{align}
\label{mikipoly}
&\sum_{k=1}^{n-1}\frac{B_k(x)B_{n-k}(y)}{k(n-k)}-\sum_{k=1}^{n-1}\binom{n-1}{k-1}\frac{B_k(x-y)B_{n-k}(y)+B_{k}(y-x)B_{n-k}(x)}{k^2}\notag\\
=&H_{n-1}\frac{B_n(x)+B_n(y)}{n}+\frac{B_n(x)-B_n(y)}{n(x-y)},
\end{align}
and
\begin{align}
\label{matiyasevichpoly}
&\sum_{k=0}^{n}B_k(x)B_{n-k}(y)-\sum_{k=0}^{n}\binom{n+1}{k+1}\frac{B_{k}(x-y)B_{n-k}(y)+B_{k}(y-x)B_{n-k}(x)}{k+2}\notag\\
=&\frac{B_{n+1}(x)+B_{n+1}(y)}{(x-y)^2}-\frac{2}{n+2}\cdot\frac{B_{n+2}(x)-B_{n+2}(y)}{(x-y)^3}.
\end{align}

With the help of some symmetric identities on Bernoulli polynomials
given in \cite{Sun03b}, they also proved a polynomial-type extension
of an identity due to Woodock \cite{Woodcock79}:
\begin{equation}
A_{m-1,n}(x)=A_{m,n-1}(x)
\end{equation}
for positive integers $m,n$, where
$$
A_{m,n}(x)=\frac{1}{n}\sum_{k=0}^n\binom{n}{k}(-1)^kB_{m+k}(x)B_{n-k}(2x)-\frac{1}{n}B_m(x)B_n(x).
$$

Subsequently, Sun and Pan \cite{SunPan06} discovered the following
symmetric identity as a generalization of the above identities:
\begin{equation}
\label{sunpan} r\left[\begin{matrix}s&t\\
x&y\end{matrix}\right]_n+ s\left[\begin{matrix}t&r\\
y&z\end{matrix}\right]_n+ t\left[\begin{matrix}r&s\\
z&x\end{matrix}\right]_n=0
\end{equation}
provided that $r+s+t=n$ and $x+y+z=1$, where
$$
\left[\begin{matrix}s&t\\
x&y\end{matrix}\right]_n=\sum_{k=0}^n(-1)^k\binom{s}{k}\binom{t}{n-k}B_{n-k}(x)B_k(y).
$$

Motivated by the results of Dilcher \cite{Dilcher96}, the referee of
\cite{SunPan06} asked whether there exists a generalization of
(\ref{sunpan}) involving sums of products of more Bernolli
polynomials. In this paper, we shall give such a generalization.
\begin{Thm} \label{t1} Let $m$ and $n$ be positive
integers. Suppose that $r_1,\ldots,r_m$ are arbitrary complex
numbers. Then
\begin{align}
\label{bernoulli}
&r_{m+1}\sum_{\substack{k_1,\cdots,k_m\geq 0\\
k_1+\cdots+k_m=n}}\prod_{j=1}^m\binom{r_j}{k_j}B_{k_j}(x_j)\notag\\
=&-\sum_{i=1}^mr_i\sum_{\substack{k_1,\ldots,k_m\geq 0\\
k_1+\cdots+k_m=n}}\binom{r_{m+1}}{k_i}B_{k_i}(1-x_i)\prod_{\substack{1\leq
j\leq m\\j\not=i}}\binom{r_j}{k_j}B_{k_j}(x_j-x_i+\1_{j>i}),
\end{align}
where $r_{m+1}=n-r_1-\cdots-r_m$ and $\1_{j>i}=1$ or $0$ according
to whether $j>i$.
\end{Thm}

Let us explain why (\ref{bernoulli}) is equivalent to
(\ref{sunpan}) when $m=2$. It is not difficult to check that
$B_k(1-x)=(-1)^kB_k(x)$. Hence in view of (\ref{bernoulli}),
\begin{eqnarray*}
\lefteqn{\sum_{k=0}^n\binom{s}{k}B_k(1-y)\binom{t}{n-k}B_{n-k}(x)}\\
&=&-s\sum_{k=0}^n\binom{r}{k}B_k(1-(1-y))\binom{t}{n-k}B_{n-k}(x-(1-y)+1)\\
&&-t\sum_{k=0}^n\binom{r}{k}B_k(1-x)\binom{s}{n-k}B_{n-k}((1-y)-x)\\
&=&-s\sum_{k=0}^n(-1)^k\binom{t}{k}B_{k}(1-x-y)\binom{r}{n-k}B_{n-k}(y)\\
&&-t\sum_{k=0}^n(-1)^k\binom{r}{k}B_k(x)\binom{s}{n-k}B_{n-k}(1-x-y),
\end{eqnarray*}
which is indeed \eqref{sunpan} by setting $z=1-x-y$.

\section{Proof of Theorem \ref{t1}}
\setcounter{Lem}{0}\setcounter{Thm}{0}\setcounter{Cor}{0}
\setcounter{equation}{0}

For a power series $f(t_1,\ldots,t_m)$, let $[t_1^{n_1}\cdots
t_m^{n_m}]f(t_1,\ldots,t_m)$ denote the coefficient of
$t_1^{n_1}\cdots t_m^{n_m}$ in $f(t_1,\ldots,t_m)$. The following
lemma is a generalization of an identity due to Carlitz \cite[Eq.
(7)]{Carlitz59}:
\begin{Lem}
\label{l1}
\begin{multline}
\sum_{i=1}^mn_iB_{n_i-1}(x_i)\prod_{\substack{1\leq j\leq m\\j\not=i}}B_{n_j}(x_j)\notag\\
=\sum_{i=1}^mn_i\sum_{\substack{k_1,\ldots,k_m\geq 0\\
k_1+\cdots+k_m=n_1+\cdots+n_m}}B_{k_i-1}(x_i)\prod_{\substack{1\leq
j\leq m\\j\not=i}}\binom{n_j}{k_j}B_{k_j}(x_j-x_i+\1_{j>i}).
\end{multline}
\end{Lem}
\begin{proof}
Consider
\begin{eqnarray}
\label{l1e1}
\lefteqn{(t_1+\cdots+t_m)\prod_{i=1}^m\frac{t_je^{x_jt_j}}{e^{t_j}-1}}\notag\\
&=&\frac{(t_1+\cdots+t_m)}{e^{t_1+\cdots+t_m}-1}\bigg(\sum_{i=1}^m(e^{t_i}-1)e^{\sum_{i<j\leq
m}t_j}\bigg)\prod_{j=1}^m\frac{t_je^{x_jt_j}}{e^{t_j}-1}\notag\\
&=&\sum_{i=1}^m\frac{t_i(t_1+\cdots+t_m)e^{x_i(t_1+\cdots+t_m)}}{e^{t_1+\cdots+t_m}-1}\prod_{\substack{1\leq j\leq m\\
j\not=i}}\frac{t_je^{(x_j-x_i+\1_{j>i})t_j}}{e^{t_j}-1}.
\end{eqnarray}
Clearly, we have
\begin{equation*}
[t_1^{n_1}\cdots
t_m^{n_m}](t_1+\cdots+t_m)\prod_{i=1}^m\frac{t_je^{x_jt_j}}{e^{t_j}-1}
=\sum_{i=1}^m\frac{B_{n_i-1}(x_i)}{(n_i-1)!}\prod_{\substack{1\leq
j\leq m\\j\not=i}}\frac{B_{n_j}(x_j)}{n_j!}.
\end{equation*}
Now, for each $1\leq i\leq m$,
\begin{align*}
&[t_1^{n_1}\cdots t_m^{n_m}]\frac{t_i(t_1+\cdots+t_m)e^{x_i(t_1+\cdots+t_m)}}{e^{t_1+\cdots+t_m}-1}\prod_{\substack{1\leq j\leq m\\
j\not=i}}\frac{t_je^{(x_j-x_i+\1_{j>i})t_j}}{e^{t_j}-1}\\
=&\sum_{\substack{k_1,\cdots,k_{i-1},k_{i+1},\cdots,k_m\geq
0}}\frac{B_{k_1+\ldots+k_{i-1}+k_{i+1}+\ldots+k_m+n_i-1}(x_i)}{k_1!\cdots
k_{i-1}!(n_i-1)!k_{i+1}!\cdots k_m!}
\prod_{\substack{1\leq j\leq m\\
j\not=i}}\frac{B_{n_j-k_j}(x_j-x_i+\1_{j>i})}{(n_j-k_j)!}.
\end{align*}
Equating the coefficients of $t_1^{n_1}\cdots t_m^{n_m}$ on both
sides of \eqref{l1e1} gives the desired identity.
\end{proof}
\begin{proof}[Proof of Theorem \ref{t1}] Applying Lemma \ref{l1}, we have
\begin{align*}
&(n-r_1-\cdots-r_m)\sum_{\substack{k_1,\cdots,k_m\geq 0\\
k_1+\cdots+k_m=n}}\prod_{j=1}^m\binom{r_j}{k_j}B_{k_j}(x_j)\\
=&-\sum_{i=1}^m(k_i+1)\sum_{\substack{k_1,\cdots,k_m\geq 0\\
k_1+\cdots+k_m=n}}
\binom{r_i}{k_i+1}B_{k_i}(x_i)\prod_{\substack{1\leq j\leq m\\ j\not=i}}\binom{r_j}{k_j}B_{k_j}(x_j)\\
=&-\sum_{\substack{k_1,\cdots,k_m\geq 0\\
k_1+\cdots+k_m=n+1}}
\prod_{j=1}^m\binom{r_j}{k_j}\bigg(\sum_{i=1}^mk_iB_{k_i-1}(x_i)\prod_{\substack{1\leq j\leq m\\ j\not=i}}B_{k_j}(x_j)\bigg)\\
=&-\sum_{\substack{k_1,\cdots,k_m\geq 0\\
k_1+\cdots+k_m=n+1}}
\prod_{j=1}^m\binom{r_j}{k_j}\sum_{i=1}^mk_i\sum_{\substack{l_1,\ldots,l_m\geq 0\\
l_1+\cdots+l_m=n}}B_{l_i}(x_i)\prod_{\substack{1\leq j\leq
m\\j\not=i}}\binom{k_j}{l_j}B_{l_j}(x_j-x_i+\1_{j>i})\\
=&
-\sum_{i=1}^mr_i\sum_{\substack{l_1,\ldots,l_m\geq 0\\
l_1+\cdots+l_m=n}}B_{l_i}(x_i)\prod_{\substack{1\leq j\leq
m\\j\not=i}}\binom{r_j}{l_j}B_{l_j}(x_j-x_i+\1_{j>i})\\
&\cdot\sum_{\substack{k_1,\cdots,k_m\geq 0\\
k_1+\cdots+k_m=n+1}}\binom{r_i-1}{k_i-1}\prod_{\substack{1\leq
j\leq m\\ j\not=i}}\binom{r_j-l_j}{k_j-l_j}.
\end{align*}
By the Chu-Vandermonde identity, we have
\begin{align*}
&\sum_{\substack{k_1,\cdots,k_m\geq 0\\
k_1+\cdots+k_m=n+1}}\binom{r_i-1}{k_i-1}\prod_{\substack{1\leq
j\leq m\\ j\not=i}}\binom{r_j-l_j}{k_j-l_j}\\
=&\binom{r_1+\cdots+r_m-1-l_1-\cdots-l_{i-1}-l_{i+1}-\cdots-l_m}
{k_1+\cdots+k_m-1-l_1-\cdots-l_{i-1}-l_{i+1}-\cdots-l_m}\\
=&\binom{r_1+\cdots+r_m-1-n+l_i}
{l_i}\\
=&(-1)^{l_i}\binom{n-r_1-\cdots-r_m} {l_i}.
\end{align*}
This completes the proof.
\end{proof}

\section{A generalization of Dunne and Schubert's identity}

 In \cite{DunneSchubert04}, Dunne and Schubert also proposed
an extension of Miki's identity (\ref{miki}) involving the gamma
function $\Gamma(z)$:
\begin{align*}
&\frac{1}{\Gamma(2p+2n)}\sum_{k=1}^{n-1} B_{2k}B_{2n-2k}
\frac{\Gamma(p+2k)\Gamma(p+2n-2k)}{\Gamma(2k+1)\Gamma(2n-2k+1)}\\
=&2\Gamma(p+1)\sum_{k=1}^n \frac{B_{2k}B_{2n-2k}}{(2k)!(2n-2k)!}
\frac{\Gamma(p+2k)\Gamma(2p+2n-1)}{\Gamma(2p+2k+1)}+\frac{2B_{2n}}{(2n)!}\sum_{k=1}^{2n-1}\beta(p+k,p+1),
\end{align*}
where $\beta(a,b)=\Gamma(a)\Gamma(b)/\Gamma(a+b)$ is the beta
function. However, applying Lemma \ref{l1} and the identity
$$
\frac{\Gamma(p+k)}{\Gamma(k+1)}=(-1)^k\Gamma(p)\binom{-p}{k}
$$
for $p\not\in\{0,-1,-2,\ldots\}$, we are led to the following
generalization.
\begin{Thm}
\label{t2} Let $m$ and $n$ be positive integers. Suppose that
$p_1,\ldots,p_m$ are non-integral complex numbers. Then
\begin{align}
\label{bernoulliNN}
&\Gamma(p_{m+1}+1)\sum_{\substack{k_1,\cdots,k_m\geq 0\\
k_1+\cdots+k_m=n}}\prod_{j=1}^m\frac{\Gamma(p_j+k_j)}{\Gamma(k_j+1)}B_{k_j}(x_j)\notag\\
=&-\sum_{i=1}^m\Gamma(p_i+1)\sum_{\substack{k_1,\ldots,k_m\geq 0\\
k_1+\cdots+k_m=n}}\frac{\Gamma(p_{m+1}+k_i)}{\Gamma(k_i+1)}B_{k_i}(1-x_i)\prod_{\substack{1\leq
j\leq
m\\j\not=i}}\frac{\Gamma(p_j+k_j)}{\Gamma(k_j+1)}B_{k_j}(x_j-x_i+\1_{j>i}),
\end{align}
where $p_{m+1}=-(p_1+\cdots+p_m+n)$.
\end{Thm}

\end{document}